\documentclass[reqno,11pt,oneside]{amsart}

\usepackage [T1]{fontenc}
\usepackage[english]{babel}
\usepackage[utf8]{inputenc}
\usepackage{soul}
\usepackage{graphicx}
\usepackage{enumitem} 
\usepackage{amsmath}
\usepackage{mathrsfs}
\usepackage{braket}
\usepackage{mathtools} 
\usepackage{amsmath} 
\usepackage{amsfonts} 
\usepackage{amsthm} 
\usepackage{geometry}
\usepackage{float}
\usepackage{multicol}
\usepackage{amssymb}
\usepackage[dvipsnames]{xcolor}

\usepackage{hyperref}

\usepackage{appendix}

\newtheorem{theorem}{Theorem}[section]
\newtheorem{corollary}{Corollary}[section]
\newtheorem{lemma}{Lemma}[section]

\theoremstyle{definition}

\theoremstyle{definition}
\newtheorem{remark}{Remark}[section]

\theoremstyle{definition}

\newcommand{\diam}{{\rm diam}}
\newcommand{\dist}{{\rm dist}}

\title[]{Quantitative results for fractional overdetermined problems in exterior and annular sets}

\author{Giulio Ciraolo}
\address{G. Ciraolo. Dipartimento di Matematica "Federigo Enriques",
Universit\`a degli Studi di Milano, Via Cesare Saldini 50, 20133 Milano, Italy}
\email{giulio.ciraolo@unimi.it}

\author{Luigi Pollastro}
\address{L. Pollastro. Dipartimento di Matematica "Federigo Enriques",
Universit\`a degli Studi di Milano, Via Cesare Saldini 50, 20133 Milano, Italy}
\email{luigi.pollastro@unimi.it}

\begin{document}

\begin{abstract}
We consider overdetermined problems related to the fractional capacity. In particular we study $s$-harmonic functions defined in unbounded exterior sets or in bounded annular sets, and having a level set parallel to the boundary.
We first classify the solutions of the overdetermined problems, by proving that the domain and the solution itself are radially symmetric. Then we prove a quantitative stability counterpart of the symmetry results: we assume that the overdetermined condition is slightly perturbed and we measure, in a quantitative way, how much the domain is close to a symmetric set.
\end{abstract}

\maketitle

\section{Introduction}
In the present paper we prove quantitative symmetry results for overdetermined problems involving the fractional Laplacian in unbounded exterior sets or bounded annular sets. These problems originate from the study of capacity of a set and relative capacity which, in the classical setting, are given by
\begin{equation*}
    {\rm cap}(\Omega) := \inf \left\{ \frac12 \int_{\mathbb{R}^n}|\nabla v|^2 dx\ :\ v\in C^{\infty}_c(\mathbb{R}^n),\,\,v{|_\Omega} \geq 1 \right\}\,,
\end{equation*}
and
\begin{equation*}
{\rm cap}(\Omega;D) := \inf \left\{ \frac{1}{2} \int_{D} |\nabla v|^2 dx:\ v \in C_c^\infty (\Omega),\ v_{|_D} \geq 1\right\},
\end{equation*}
respectively; here $D$ and $\Omega$ are bounded open sets, with $\overline D \subset \Omega \subset \mathbb{R}^n$, $n \geq 3$, and $\nabla v$  is the gradient of the function $v$.

Instead of the classical notion of capacity, in this paper we consider the capacity in a fractional setting. For a parameter $s \in (0,1)$, the $\textit{fractional capacity of order s}$ (or $\textit{s-capacity}$) of the set $\Omega$ is defined as follows:
\begin{equation} \label{caps_def}
    {\rm cap}_s(\Omega) := \inf\{ [v]_s^2 \ | \ v \in C_c^\infty (\mathbb{R}^n),\ v_{|_\Omega} \geq 1\} \,,
\end{equation}
where $[v]_s$ is the $\textit{Gagliardo seminorm}$ of $v$ which is defined by
\begin{equation*}
    [v]_s^2 := \int_{\mathbb{R}^{2n}} \frac{|v(x)-v(y)|^2}{|x-y|^{n+2s}} dx dy \,.
\end{equation*}
Analogously, one can define the $\textit{relative fractional capacity of order s}$ of the couple of sets $(\Omega,D)$ by
\begin{equation} \label{caps_rel_def}
    {\rm cap}_s(\Omega;D) := \inf\{ [v]_s^2 \ | \ v \in C_c^\infty (\Omega),\ v_{|_D} \geq 1\} \,.
\end{equation}

The Euler-Lagrange equations associated to \eqref{caps_def} and \eqref{caps_rel_def} are both related to the so-called fractional Laplacian of order $s$ (or $s$-Laplacian), which is denoted by $(-\Delta)^s$ and it is given by  
\begin{equation*}
    (-\Delta)^s u(x) := c_{n,s} \, P.V. \int_{\mathbb{R}^n} \frac{u(x) - u(z)}{|x-z|^{n+2s}} dz,
\end{equation*} 
for $u \in C^{\infty}_c (\mathbb{R}^n)$, where
\begin{equation} \label{cns}
    c_{n,s} = s \, (1-s) \, 4s \pi^{-n/2} \frac{\Gamma(n/2 + s)}{\Gamma(2-s)} 
\end{equation}
(see for example \cite{di2012hitchhiker}). 
It can be proved that ${\rm cap}_s(\Omega)$ and  ${\rm cap}_s(\Omega;D)$ are uniquely achieved by two functions $u_\Omega,\, u_{\Omega,D} \in  H^s(\mathbb{R}^n)$ which satisfy
\begin{equation}
\label{a2s1eq10}
    \begin{cases}
    (-\Delta)^s u_\Omega = 0 \qquad &\text{in} \ \mathbb{R}^n \setminus \overline\Omega,\\
    u_\Omega = 1 \qquad &\text{in} \ \overline \Omega,\\
    u_\Omega(x) \to 0 \qquad &\text{as} \ |x| \to +\infty\,,
    \end{cases}
\end{equation}
and
\begin{equation}
\label{a2s1eq12}
    \begin{cases}
    (-\Delta)^s u_{\Omega,D} = 0 \qquad &\text{in} \ A:= \Omega \setminus \overline D,\\
    u_{\Omega,D} = 1 \qquad &\text{in} \ \overline D,\\
    u_{\Omega,D} = 0 \qquad &\text{in} \ \mathbb R^n \setminus \Omega \,,
    \end{cases}
\end{equation}
respectively. The function $u_\Omega$ is sometimes called the \textit{$s$-capacitary potential}.

Overdetermined problems for \eqref{a2s1eq10} and \eqref{a2s1eq12} have been considered in \cite{soave2019overdetermined} where the overdetermined condition is given on the normal \emph{$s$-derivative} at the boundary, which is assumed to be constant in the spirit of Serrin's overdetermined problem. 

In this paper we consider a somehow discrete version of Serrin's overdetermined condition, and we instead assume that the solution is constant on a surface parallel to the boundary.\footnote{Regarding problem \eqref{a2s1eq12}, it is more precise to say that the solution is constant on each connected component of the parallel surface, see Theorem \ref{a2theorem3} below.} In this setting, our main results can be considered as the generalization of the results in \cite{ciraolo2021symmetry} to exterior and annular domains.

In order to clearly state our results, we recall that the Minkowski sum of two sets $A$ and $B$ is defined by
$$
A+B = \{ x + y \ | \ x \in A \ \, y \in B \}.
$$

%

Our first result deals with solutions of problem \eqref{a2s1eq10} with the overdermining assumption that the solution is constant on a surface parallel to $\partial \Omega$.

%

\begin{theorem}
\label{a2s1theorem1}
Let $\Omega$ be a bounded domain in $\mathbb{R}^n$. Let $R>0$ and assume that  $G:= \Omega + B_R$ is such that $\partial G$ of class $C^1$. Then, there exists a solution $u \in H^s(\mathbb{R}^n) \cap C(\mathbb{R}^n)$ of \eqref{a2s1eq10} such that 
\begin{equation}\label{a2s1eq16}
    u = c \qquad \text{on} \ \partial G
\end{equation}
for some constant $c$ if and only if $G$ and $\Omega$ are concentric balls and $u$ is radially symmetric.
\end{theorem}

We will prove Theorem \ref{a2s1theorem1} by using the method of moving planes. Once symmetry is established, one can investigate the quantitative stability result for Theorem \ref{a2s1theorem1}. The idea is to assume that the overdetermined condition \eqref{a2s1eq16} is replaced by a weaker condition which implies that the solution is close to a constant on $\partial G$. In this direction, it is useful to consider the Lipschitz seminorm $[u]_{\Gamma}$ of $u$ on $\Gamma=\partial G$, which is given by
\begin{equation*}
    [u]_{\Gamma} := \sup_{x,y \in \Gamma, \, x \neq y} \frac{|u(x) - u(y)|}{|x-y|}
\end{equation*}
and the parameter
\begin{equation}\label{def:outradius-inradius}
    \rho (\Omega) := \inf \{ |t - s| \ | \ \exists p \in \Omega \ \mathrm{such \ that} \ B_s(p) \, \subset \Omega \subset B_t(p) \} \,,
\end{equation}
which controls how much the set $\Omega$ differs from a ball (clearly, $\rho(\Omega) = 0$ if and only if $\Omega$ is a ball). 

Another relevant quantity which we need to quantify the stability results is the radius of the touching ball condition. More precisely, given a set $E$ we denote the optimal exterior and interior radii in the touching ball condition by $\mathfrak{r}_{E}^e$ and $\mathfrak{r}_{E}^i$, respectively.
 
Hence, our main goal is to obtain quantitative bounds on $\rho (\Omega)$ in terms of $[u]_{\partial G}$, as done in the following theorem.

\begin{theorem}
\label{a2theorem2}
Let $\Omega$ be a bounded domain of $\mathbb{R}^n$ with $\partial \Omega$ of class $C^2$. Let $R>0$ and let $G =\Omega + B_R$ be such that $\partial G$ is of class $C^2$. Let $u \in C^s(\mathbb{R}^n)$ be a solution of \eqref{a2s1eq10}. Then, we have that
\begin{equation}
\label{a2s1eq8}
\rho (\Omega) \leq C_* \, [u]_{\partial G}^{\frac{1}{s +2}} ,
\end{equation}
with $C_* =C_*(n,s,R,\mathrm{diam}(\Omega),|\Omega|,\mathfrak{r}_{\Omega}^e)> 0$, where $\mathrm{diam}(\Omega)$ and $|\Omega|$ denote the diameter and the volume of $\Omega$, respectively, and $\mathfrak{r}_{\Omega}^e$ is the radius of the exterior touching ball condition at $\Omega$. 
\end{theorem}

In the second part of the article we consider an overdetermined problem involving \textit{annular sets}. More precisely, let $D,\Omega \subset \mathbb{R}^n$ be bounded open domains such that $\overline D \subset \Omega$, set
\begin{equation} \label{A_def}
A:= \Omega \setminus \overline{D} \,,
\end{equation}
and we consider solutions to \eqref{a2s1eq12}.
It is clear that, since $\partial \Omega$ and $\partial D$ do not touch, we have that 
\begin{equation} \label{dbar}
\overline{d}:= \dist (\overline{D}, \mathbb{R}^n \setminus \Omega) >0 \,.
\end{equation}
By choosing a positive parameter $R < \overline{d}/2$ we have that the set 
\begin{equation} \label{Gamma_RA}
\Gamma_R^A := \{ x \in A \ | \ \dist (x,\partial A) = R \}
\end{equation}
can be written as
\begin{equation} \label{Gamma_union}
\Gamma_R^A = \Gamma_R^D \cup \Gamma_R^\Omega\,,
\end{equation}
with \footnote{Notice that $\Gamma_R^A = \partial ( (\Omega^c + B_R) \setminus (D+B_R))$.}
\begin{align*}
    \Gamma_R^D := \{ x \in A \ | \ \dist(x,\partial D) = R \},\\
    \Gamma_R^\Omega := \{ x \in A \ | \ \dist(x,\partial \Omega) = R \},
\end{align*}
with $\Gamma_R^D  \cap  \Gamma_R^\Omega = \emptyset$. On each of these hypersurfaces we assume that the solution satisfies the overdetermined condition
\begin{equation}
\label{a2s1eq15}
    \begin{aligned}
    u = \alpha \qquad &on \ \Gamma_R^D,\\
    u = \beta \qquad &on \ \Gamma_R^\Omega,
\end{aligned}
\end{equation}
where $\alpha$ and $\beta$ are two positive constants.

\bigskip

We have the following symmetry result. 

\begin{theorem}
\label{a2theorem3}
Let $A$ and $\Gamma^A_R$ be given by \eqref{A_def} and \eqref{Gamma_RA}, respectively, where $R$ is such that $\Gamma^A_R$ is of class $C^1$. 

Let $u \in H^s(\mathbb{R}^n) \cap C(\mathbb{R}^n)$ be a solution of \eqref{a2s1eq12} satisfying the overdetermined conditions \eqref{a2s1eq15}. 
Then, $D$ and $\Omega$ are concentric balls and $u$ is radially symmetric.
\end{theorem}

Now we describe the quantitative stability result that we obtain for Theorem \ref{a2theorem3}. In this case, we replace the overdetermined condition \eqref{a2s1eq12} by assuming that the solution has small Lipschitz seminorm on each connected component of $\Gamma_R^A$. For this reason we define the following deficit 
\begin{equation} \label{deficit_A}
    \mathrm{def}_A(u):= \max \{ [u]_{\Gamma_R^D}, [u]_{\Gamma_R^\Omega} \},
\end{equation}
and we have the following result.

\begin{theorem}
\label{a2theorem4}
Let $A$ and $\Gamma^A_R$ be given by \eqref{A_def} and \eqref{Gamma_RA}, respectively, and assume that $\partial A$ and $\Gamma^A_R$ are of class $C^2$.

Let $u \in C^s(\mathbb{R}^n)$ be a solution of \eqref{a2s1eq12}. Then
\begin{equation}
\label{a2s1eq18}
\rho(D) + \rho(\Omega) \leq C_* \mathrm{def}_A(u)^{\frac{1}{s +2}} \,,
\end{equation}
with $\rho$ given by \eqref{def:outradius-inradius} and $C_*=C_*(n,s,R,\diam(\Omega), |\Omega|, |D|, \mathfrak{r}^e_{D}, \mathfrak{r}^i_{\Omega}) > 0$, where $\mathfrak{r}^e_{D}$ and $\mathfrak{r}^i_{\Omega}$ are the radius of the uniform exterior touching ball to $D$ and of the interior touching ball to $\Omega$, respectively. 
\end{theorem}

Theorems \ref{a2theorem2} and \ref{a2theorem4} are the main results of this paper, and they are obtained by using a quantitative approach to the method of moving planes, which was originally developed in \cite{aftalion1999approximate} (see also \cite{CiraoloVezzoni}, \cite{ciraolo2016solutions}, \cite{CMS}, \cite{ciraolo2018rigidity}, \cite{ciraolo2021symmetry}). This approach presents many differences when applied in a classical local settings and in a fractional framework, and in this paper we prefer to tackle fractional problems. We mention that all the symmetry and quantitative symmetry results in this paper have their classical local counterpart, which can be still attacked by using the method of moving planes and it will be considered in a future work. 

We finally notice that, in Theorems \ref{a2s1theorem1}-\ref{a2theorem4}, we assumed that $\Omega$ and $D$ are domains and then they are connected. The connectedness assumption is not necessary and it can be easily removed, and hence our results can be extended in that setting. However, this has a cost in managing the notation and it would worsen the presentation and clarity of the paper. For this reason, we preferred to assume that $\Omega$ and $D$ are connected.

\medskip

The paper is organized as follows. In Section 2 we present some preliminary notions and results, including a weak maximum principle for $s$-harmonic functions in an unbounded domain. Section 3 is devoted to the results for exterior sets and includes the standard machinery for the method of moving planes. In Section 4 we consider the problems involving annular domains.

\subsection*{Acknowledgements}
The authors have been partially supported by the ``Gruppo Nazionale per l'Analisi Matematica, la Probabilit\`a e le loro Applicazioni'' (GNAMPA) of the ``Istituto Nazionale di Alta Matematica'' (INdAM, Italy). The authors thank the anonymous referee for the valuable comments and suggestions.

%
%
%
%
%
%
%
%
%

\section{Preliminaries and notation}

In this section we introduce some notation and recall some results which will be useful in the rest of the paper.

We recall that given two functions $u,v \in H^{s}(\mathbb{R}^n)$ the $\textit{Gagliardo seminorm}$ of $u$ is defined as
\begin{equation*}
    [u]_s^2 := \frac{c_{n,s}}{2} \int_{\mathbb{R}^n} \int_{\mathbb{R}^n} \frac{|u(x)-u(y)|^2}{|x-y|^{n+2s}} dx dy,
\end{equation*}
and the scalar product in $H^s(\mathbb{R}^n)$ between $u$ and $v$ is defined as
\begin{equation*}
    \mathcal{E}(u,v) = \frac{c_{n,s}}{2} \int_{\mathbb{R}^n}\int_{\mathbb{R}^n}\frac{(u(x)-u(y))(v(x)-v(y))}{|x-y|^{n+2s}} dx dy \,,
\end{equation*}
where $c_{n,s}$ is given by \eqref{cns}.


%

In order to write a Hopf's boundary point lemma in a quantitative form, it is useful to consider the solution $\psi_{B_r(x_0)}$ to the fractional torsion problem in a ball of radius $r>0$ centered $x_0$, i.e. $\psi_{B_r(x_0)}$ satisfies
\begin{equation} \label{torsion_pb} 
\begin{cases}
(-\Delta)^s \psi_{B_r(x_0)} = 1 & \text{ in } B_r(x_0) \\
\psi_{B_r(x_0)} = 0 & \text{ in } \mathbb{R}^n \setminus B_r(x_0)\,,
\end{cases}
\end{equation}
and it is given by 
\begin{equation} \label{psi_def}
\psi_{B_r(x_0)}(x):= \gamma_{n,s} (r^2 - |x-x_0|^2)^s_+
\end{equation}
for any $x \in \mathbb{R}^n$, where $\gamma_{n,s}$ is a constant depending only on $n$ and $s$.

\begin{lemma}
\label{a2s3theorem1}
Let $\Omega \subset \mathbb{R}^n$ be an open set and let $u \in H^s(\mathbb{R}^n)$ be a solution of
\begin{equation*}
    \begin{cases}
    (-\Delta)^s u \geq 0 \quad & \text{in} \ \Omega,\\
    u \geq 0 \quad &\text{in} \ \mathbb{R}^n \setminus \Omega \,.
    \end{cases}
\end{equation*}
Let $x_0 \in \Omega$ and $r>0$ be such that $B_r(x_0) \subseteq \Omega$. Let $K\subset \mathbb{R}^n$ be a compact set such that 
$$
|K| > 0 \,, \quad \dist(K,B_r(x_0)) > 0 \,,\quad  \rm{essinf}_K u > 0 \,.
$$ 
Then \begin{equation*}
    u \geq C_H \, \psi_{B_r(x_0)} \qquad \text{in} \  \, B_r(x_0)\,,
\end{equation*}
where 
\begin{equation} \label{CHopf}
   C_H := c_{n,s} \frac{|K| \, \rm{essinf}_K u}{(2r + \dist(K,B_r(x_0)) + \diam(K))^{n+2s}} \,,
\end{equation}
with $c_{n,s}$ and $\psi_{B_r(x_0)}$ given by \eqref{cns} and \eqref{psi_def}, respectively.
\end{lemma}

Lemma \ref{a2s3theorem1} was already proved in \cite{greco2016hopf} and \cite{ros2014dirichlet}. Here, inspired by  \cite{fall2015overdetermined} and \cite{ciraolo2021symmetry}, we give a proof which allows us to explicitly write the constant $C_H$ given by \eqref{CHopf}, and to show its dependency on the parameters which are relevant in our problem. This will be useful when we will prove the quantitative results.

\begin{proof}[Proof of Lemma \ref{a2s3theorem1}] 
We consider the barrier function 
$$
w(x) := \psi_B (x) + \delta \, \chi_K (x) \,,
$$ 
where $B=B_r(x_0)$, $\chi_K$ is the characteristic function of $K \subset \mathbb{R}^n$ and $\delta > 0$ is a constant that will be chosen later.

Let $\varphi \in H^s_0 (\Omega)$ be a nonnegative test function. We have
\begin{align*}
    \mathcal{E} (w,\varphi) &= \mathcal{E} (\psi_B, \varphi) + \delta \, \, \mathcal{E} (\chi_K, \varphi) = \int_B \varphi - \delta \, c_{n,s} \, \int_K \int_B \frac{\varphi(y)}{|x - y|^{n+2s}} \, dy \, dx \leq \\
    &\leq (1 -\delta \, C) \int_B \varphi,
\end{align*}
which is less or equal than zero if we choose $\delta \geq C^{-1}$ with
\begin{equation*}
    C  = c_{n,s} \, |K| \, \inf_{x \in K, y \in B} \frac{1}{|x - y|^{n+2s}}. 
\end{equation*}

By setting 
$$
\tau := \rm{essinf}_K u / \delta = C \, \rm{essinf}_K u
$$ 
and applying the weak maximum principle for $s$-harmonic functions to 
$$
v := u - \tau \, w \,,
$$ 
we get that
\begin{equation*}
    u \geq c_{n,s} \frac{|K| \, \rm{essinf}_K u}{(\diam(B) + \dist(K,B) + \diam(K))^{n+2s}} \, \psi_B \quad \text{in} \, B,
\end{equation*}
which is the desired result. 
\end{proof}

Since our approach is based on the method of moving planes, a particular attention must be given to antisymmetric $s$-harmonic functions. More precisely, we will have to consider functions which are antisymmetric with respect to a hyperplane which can be chosen to be $\{x_1 = 0\}$ (up to a translation and rotation). 

In order to list these results, we need to introduce some notation: we set $H^+ := \{x_1 > 0\}$, $H^- := \{ x_1 < 0 \}$ and $T\coloneqq \{ x_1 = 0 \}$. Let 
$$
\mathcal Q: \mathbb{R}^n \to \mathbb{R}^n\,, \ \  y \mapsto y' = (-y_1,y_2,\dots,y_n) \,,
$$ 
be the reflection with respect to $T$ and, for a given set $E$ we call $E^+:= E \cap H^+$ and $E^-:= E \cap H^-$.

The first result is a weak maximum principle for $s$-harmonic antisymmetric functions, which is stated in \cite[Proposition 3.1]{fall2015overdetermined} on domains that are bounded, although for homogeneous equations this condition is not needed. We report this proposition here and we sketch a proof.

\begin{lemma}[Weak maximum principle for antisymmetric functions]
\label{a2s1lemma1}
Let $\Omega \subset \mathbb{R}^n$ be a compact set and let $u \in H^s(\mathbb{R}^n)$ be an antisymmetric (w.r.t. $T=\{ x_1 = 0 \}$) solution of
\begin{equation*}
    \begin{cases}
    (-\Delta)^s u = 0 \qquad &\text{in} \ \Omega^c,\\
    u \geq 0 \qquad &\text{in} \ \Omega^+.
    \end{cases}
\end{equation*}
Then, $u \geq 0$ a.e. in $H^+$.
\end{lemma}

\begin{proof}
Since $u$ is $s$-harmonic in $\Omega^c$ then for every $ \varphi \in H^s_0(\Omega^c)$ we have
\begin{equation}
    \mathcal{E}(u,\varphi) = 0.
\end{equation}

Let $\varphi = u_- \chi_{H^+} \in H^s_0(\Omega^c)$, where $\chi_{H^+}$ is the characteristic function of $H^+$. Following the same computations as in \cite[Proposition 3.1]{fall2015overdetermined} we get 
\begin{equation}
    0 \leq \mathcal{E}(u,\varphi) \leq - \mathcal{E}(\varphi,\varphi) = - [\varphi]_s^2 \,,
\end{equation}
which immediately implies that $\varphi=0$ a.e. and hence $u_-=0$ a.e. in $H^+$.
\end{proof}

An analogous weak maximum principle holds for nonnegative functions in $H^s(\mathbb{R}^n)$. More precisely we have

\begin{lemma}[Weak maximum principle]
\label{a2s1lemma1bis}
Let $\Omega \subset \mathbb{R}^n$ be a compact set and let $u \in H^s(\mathbb{R}^n)$ be a solution of
\begin{equation*}
    \begin{cases}
    (-\Delta)^s u = 0 \qquad &\text{in} \ \Omega^c,\\
    u \geq 0 \qquad &\text{in} \ \Omega.
    \end{cases}
\end{equation*}
Then, $u \geq 0$ a.e. in $\mathbb{R}^n$.
\end{lemma}

\begin{proof}
The proof is analogous to the one of Lemma \ref{a2s1lemma1}, since it is enough to consider $ \varphi = u_-$.
\end{proof}
 
As an immediate consequence we have the following comparison principle for $s$-capacitary functions.  

\begin{corollary} \label{corol_monot}
Let $E\subset F \subset \mathbb{R}^n$ be open bounded domains, and let $u_E$ and $u_F$ be the corresponding capacitary functions, i.e. the solutions to \eqref{a2s1eq10} for $\Omega=E$ and $\Omega=F$, respectively. Then we have 
\begin{equation} \label{monotonicity}
u_E \leq u_F
\end{equation}
in $\mathbb R^n$.
\end{corollary}

\begin{proof}
Since $u_E$ is a $s$-capacitary function, from \cite[Lemmas 2.6 and 2.7]{warma2015fractional} we have that $0 \leq u_E \leq 1$ in $\mathbb{R}^n \setminus E$. Then, by applying Lemma \ref{a2s1lemma1bis} to $v:= u_F - u_E$ we obtain the result.
\end{proof}

From Lemma \ref{a2s1lemma1}, we can also recover a quantitative version of the Hopf lemma for antisymmetric functions as proved in \cite{ciraolo2021symmetry}, that we recall below.

\begin{lemma} [Lemma 4.1 in \cite{ciraolo2021symmetry}]
\label{s3lemma2}
Let $\Omega$ be an open set in $H^-$ and $B \subset \Omega$ a ball of radius $R > 0$ such that $\mathrm{dist}(B,H^+) > 0$. Let $v \in H^s(\mathbb{R}^n)$ be antisymmetric and a solution of
\begin{equation*}
    \begin{cases}
    (-\Delta)^s v \ge 0 \quad & \text{in} \ \Omega,\\
    v \geq 0 \quad & \text{in} \ H^-.
    \end{cases}
\end{equation*}
Let $K \subset H^-$ be a set of positive measure such that $\overline{K} \subset (H^- \setminus \overline{B})$ and $\rm{essinf}_K v > 0$. Then we have that
\begin{equation}
\label{s3eq19}
v \geq C \big[  \mathrm{dist} (K,H^+) \, |K| \, \rm{essinf}_K v \big] \psi_B \quad \text{in} \ B,
\end{equation}
with
\begin{equation*}
    C:=  \frac{ 2 (n+2s) \, C(n,s) \, \mathrm{dist} (B,H^+)^{n+2s+1}}{(\mathrm{dist}(B,H^+)^{n+2s}+\gamma_{n,s} \, C(n,s) \, |B| \, R^{2s}) \, ( \mathrm{diam} (B) + \mathrm{diam}(K) + \mathrm{dist}(\mathcal{Q}(K),B))^{n+2s+2}}.
\end{equation*}
\end{lemma}

\begin{remark}
    Lemma 4.1 in \cite{ciraolo2021symmetry} actually requires that $v \in C^s(\Omega)$, but it is straightforward to verify that the proof is still valid if one assumes $v \in H^s(\mathbb{R}^n)$. For this reason we omit the proof of Lemma \ref{s3lemma2}.
\end{remark}

It is clear that Lemma \ref{s3lemma2} provides a quantitative version of the strong maximum principle for antisymmetric $s$-harmonic functions, which still holds when $\Omega$ is not bounded, as already noted in \cite[Proposition 2.1]{soave2019overdetermined}. 

\begin{lemma}[Strong maximum principle for antisymmetric functions] \label{lemma_SMP}
Let $\Omega$ be an open set with $\Omega \subset H^-$ and let $v \in C(\Omega)$ be antisymmetric and a solution of
\begin{equation*}
    \begin{cases}
    (-\Delta)^s v \ge 0 \quad & \text{in} \ \Omega,\\
    v \geq 0 \quad & \text{in} \ H^-.
    \end{cases}
\end{equation*}
Then, either $v > 0$ in $\Omega$ or $v \equiv 0$ in $\mathbb{R}^n$.
\end{lemma}

\begin{proof}
From the weak maximum principle in Lemma \ref{a2s1lemma1} we have that $v \geq 0$ in $\Omega$. Now assume there exists $x_0 \in \Omega$ such that $v(x_0) = 0$ and choose a ball $B$ centered in $x_0$ and such that $\overline{B} \subset \Omega$. Let $K \subset \Omega$ be a compact set  such that $\dist(B,K) > 0$ and $|K| > 0$. If we furthermore choose $B$ and $K$ such that $\inf_K v > 0$, by applying Lemma \ref{s3lemma2} we have
\begin{equation*}
v \geq C \big[  \dist (K,H^+) \, |K| \, \inf_K v \big] \psi_B \quad in \ B,
\end{equation*}
and in particular $v(x_0) > 0$, which is a contradiction.
\end{proof}

Another tool from \cite{ciraolo2021symmetry} that we will need in our proof is the boundary Harnack inequality for $s$-harmonic antisymmetric functions that we report here for clarity. 

\begin{lemma}[Lemma 2.1 in \cite{ciraolo2021symmetry}]
\label{a2s1lemma5}
Let $u\in C^2(B_R) \cap C(\mathbb{R}^n)$ be a solution of
\begin{equation*}
    \begin{cases}
    (-\Delta)^s u = 0 \quad & \text{in }  B_R,\\
    u(x') = - u(x) \quad &\text{for  every }  x \in \mathbb{R}^n,\\
    u \geq 0 \quad &  \text{in }  H^+.
    \end{cases}
\end{equation*}
There exists a constant $K > 1$ only depending on $n$ and $s$ such that, for every $z \in B_{R/2}^+$ and for every $x \in B_{R/4}(z) \cap B_R^+$ we have
\begin{equation}
\label{s3eq2}
    \frac{1}{K} \frac{u(z)}{z_1} \leq \frac{u(x)}{x_1} \leq K \frac{u(z)}{z_1}.
\end{equation}
\end{lemma}

\section{Exterior sets}
In this section we consider the exterior overdetermined problem and prove Theorems \ref{a2s1theorem1} and \ref{a2theorem2}.

\subsection{The method of moving planes: notation}  \label{subsect_MMP}
We introduce some notation in order to exploit the moving planes method. Given $e \in \mathbb{S}^{n-1}$, a set $E \subset \mathbb{R}^n$ and $\lambda \in \mathbb{R}$, we define
\begin{align*}
    &T_\lambda = T_\lambda^e = \{ x \in \mathbb{R}^n \, | \, x \cdot e = \lambda \} & &\textrm{a hyperplane orthogonal to } e,\\
    &H_\lambda = H_\lambda^e = \{ x \in \mathbb{R}^n \, | \, x \cdot e > \lambda \} & &\textrm{the ``positive'' half space with respect to } T_\lambda \\
    &E_\lambda = E \cap H_\lambda & &\textrm{the ``positive'' cap of } E,\\
    &x_\lambda' = x -2(x\cdot e - \lambda) \, e & &\textrm{the reflection of } x \textrm{ with respect to } T_\lambda,\\
    &\mathcal{Q}= \mathcal{Q}_\lambda^e : \mathbb{R}^n \to \mathbb{R}^n, x \mapsto x_\lambda' & &\textrm{the reflection with respect to } T_\lambda.
\end{align*}

If $E \subset \mathbb{R}^n$ is an open bounded set with boundary of class $C^1$ then we define 
$$
\Lambda_e := \sup \{ x \cdot e \, | \, x \in E \}
$$ 
and
\begin{equation*}
    \lambda_e = \inf \{ \lambda \in \mathbb{R} \, | \, \mathcal{Q}(E_{\Tilde{\lambda}}) \subset E, \textrm{for all} \, \Tilde{\lambda} \in (\lambda, \Lambda_e) \}.
\end{equation*}

{F}rom this point on, given a direction $e \in \mathbb{S}^{n-1}$, we will refer to $T_{\lambda_e} = T^e$ and $E_{\lambda_e} = \widehat{E}$ as the \textit{critical hyperplane} and the \textit{critical cap} with respect to $e$, respectively, and will call $\lambda_e$ the \textit{critical value} in the direction $e$. We now recall from \cite{serrin1971symmetry} that, for any given direction $e$, at least one of the following two conditions holds:\\

\textbf{Case 1} - The boundary of the reflected cap $\mathcal{Q}^e(\widehat{E})$ becomes internally tangent to the boundary of $E$ at some point $P \not \in T^e$;\\

\textbf{Case 2} - the critical hyperplane $T^e$ becomes orthogonal to the boundary of $E$ at some point $Q \in T^e$.

\subsection{The symmetry result}

We start with the symmetry result given in Theorem \ref{a2s1theorem1}. 

\begin{proof}[Proof of Theorem \ref{a2s1theorem1}]
Let $e \in \mathbb{S}^{n-1}$ be a fixed direction. Withouth loss of generality we assume $e=e_1$. We recall that we are considering a solution $u \in C^s(\mathbb{R}^n)$ of \eqref{a2s1eq10} satisfying \eqref{a2s1eq16} and that $G=\Omega + B_R$, with $\partial G$ of class $C^1$. 

We apply the method of moving planes described in Subsection \ref{subsect_MMP} by letting $E=G$. Without loss of generality, we can assume that $\lambda_e = 0$ (that is, the critical hyperplane $T$ goes through the origin), and we simplify the notation by setting $H^- := \{ x_1 < 0 \}$, $\Omega^- := H^- \cap \Omega$ and considering
\begin{equation*}
    v(x):= u(x) - u(\mathcal Q(x)) \quad \text{for } \ x \in \mathbb{R}^n,
\end{equation*}
where $\mathcal Q: \mathbb{R}^n \to \mathbb{R}^n, x \mapsto x'$ is the reflection with respect to $T$. We have
\begin{equation*}
    \begin{cases}
    (-\Delta)^s v = 0 \quad &\text{in} \ H^- \setminus \Omega^- \\
    v \geq 0 \quad &\text{in} \ \Omega^-\\
    v(\mathcal Q(x)) = - v (x) \quad &\text{for every } x \in \mathbb{R}^n.
\end{cases}
\end{equation*}

By using Lemma \ref{a2s1lemma1} we know that $v \geq 0$ in $H^-$ and then Lemma \ref{lemma_SMP} tells us that either $v > 0$ in $H^- \setminus \Omega^-$ or $v \equiv 0$ in $\mathbb{R}^n$. Now we show that if we assume that $v > 0$ in $H^-\setminus \Omega^-$ then we obtain a contradiction.

\textbf{Case 1} - Let $P$ be a critical point on $\partial G$. Since both $P$ and its reflection $P'$ belong to $\partial G$ and \eqref{a2s1eq16} holds, we immediately get
\begin{equation*}
    v(P) = u(P) - u(P') = 0,
\end{equation*}
which is a contradiction.\\

\textbf{Case 2} - In this case $e_1$ is tangent to $\partial G$ at a point $Q \in \partial G$, and therefore we have that $\partial_1 v (Q) = 0$. On the other hand, since $Q$ is far away from the boundary $\partial \Omega$, we can use Lemma \ref{a2s1lemma5} to show that $\partial_1 v (Q) < 0$, which is a contradiction.

Indeed, setting $z = (-R/4, Q_2, \dots, Q_n)$ and $x=x_t = (-t, Q_2, \dots, Q_n) \in B_{R/4}(z)$, with $0<t<R/8$, we have that
\begin{equation}
\label{a2s2eq43}
    \frac{v(x_t)}{- t} \ge - \frac{ 4}{R K} v(z),
\end{equation}
where $K>1$ is a constant only depending on $n$ and $s$.
Being $z \in H^- \setminus \Omega^-$, we have that $v(z)>0$, and the claim follows from \eqref{a2s2eq43} by letting $t\to 0^+$.
\end{proof}

\subsection{Almost symmetry in one direction}
Now we consider the quantitative stability result and prove Theorem \ref{a2theorem2}. This will be done in two subsequent steps: we first prove the quantitative stability estimate in one direction and then, in the proof of Theorem \ref{a2theorem2} we will sketch a general idea of how to use the result in one direction to obtain the final quantitative estimate; the proof can be found in details in Section 6 of \cite{ciraolo2021symmetry}. 

We start by proving a preliminary result which gives the behaviour of the solution to \eqref{a2s1eq10} close to the boundary. 

\begin{lemma} \label{lemma_conticino}
Under the assumptions of Theorem \ref{a2theorem2}, let $u$ be a solution of \eqref{a2s1eq10} and let $v:= 1- u$. For any $r \leq \mathfrak{r}_{\Omega}^e$ we have
\begin{equation}
\label{a2s3eq35}
     v(x) \geq  C_{cap} \, (\dist (x,\partial \Omega))^s \quad \text{in} \ (\Omega + B_r)\setminus \Omega,
\end{equation}
where
\begin{equation*}
    C_{cap} := \frac{c_{n,s} \, \gamma_{n,s} \, \omega_{n}}{4} \, \frac{r^{n+s}}{(2r + r_0 \, \diam (\Omega))^{n+2s}}
\end{equation*}
and $r_0 > 0$ is a constant depending on $n$ and $s$.
\end{lemma}

\begin{proof}
Without loss of generality, we can assume that the origin $O$ is contained in $\Omega$ and consider the $s$-capacitary solution $\tilde{u}$ of the ball $B_{\diam(\Omega)}$ centered at the origin and of radius $\diam(\Omega)$:  
\begin{equation*}
    \begin{cases}
    (-\Delta)^s \tilde{u} = 0 \quad &\text{in} \ \mathbb{R}^n \setminus B_{\diam(\Omega)},\\
    \tilde{u} = 1 \quad &\text{in} \ B_{\diam(\Omega)},\\
    \tilde{u}(x) \to 0 \quad &\text{as} \ |x| \to + \infty.
    \end{cases}
\end{equation*}
Since $0 \leq \tilde{u} \leq 1$ and $\tilde{u}$ is radial, non-increasing and continuous (see for instance \cite[Theorem 1.10]{soave2019overdetermined}, there exists a radius $\tilde{R} = \tilde{R}(\diam (\Omega)) > 0$ such that
\begin{equation*}
    \tilde{u} < 1/2 \quad \text{in} \quad \mathbb{R}^n \setminus B_{\tilde{R}}.
\end{equation*}
Moreover, from Corollary \ref{corol_monot} we have that $\tilde{u} \geq u$ in the whole space. From this we get that
\begin{equation}
    v = 1 - u \geq 1 - \tilde{u} \geq 1/2 \quad \text{in} \ \mathbb{R}^n \setminus B_{\tilde{R}}.
\end{equation}

We now choose $K=\overline{B_{\mathfrak{r}^e_\Omega}((\tilde{R}+\mathfrak{r}^e_\Omega) \,e_1)}$. For $x_0 \in \partial \Omega$, we now apply Lemma \ref{a2s3theorem1} to $v$ with $B = B_{\mathfrak{r}^e_\Omega}(x_0)$ and $K$ and get
\begin{equation*}
    v(x) \geq \frac{c_{n,s} \, \gamma_{n,s} \, \omega_n}{2} \, \frac{R^{n+s}}{(4R + \dist(K,B))^{n+2s}} \, (R - |x|)^s \quad \text{in} \ B.
\end{equation*}

We now repeat the same argument on the whole boundary $\partial \Omega$ by keeping each time the same fixed set $K$. We notice that in every case we have $\dist(K,B) \leq 2\tilde{R}$, and by using the previous inequality we obtain \eqref{a2s3eq35}, where the constant $C_{cap} > 0$ can be written as 
\begin{equation*}
   C_{cap} := \frac{c_{n,s} \, \gamma_{n,s} \, \omega_n}{4} \, \frac{R^{n+s}}{(2R + \tilde{R})^{n+2s}}.
\end{equation*}
In order to complete the proof, we show how $\tilde R$ depends on $\diam(\Omega)$. We consider the solution $u_{B_1}$ to the capacitary problem \eqref{a2s1eq10} with $\Omega = B_1$, and we set
$$
r_0=\inf\{ \, |x| \ | \ u_{B_1}(x)< 1/2 \} \,. 
$$ 
By scaling properties, it is clear that $\tilde R = r_0 \, \diam (\Omega)$. This completes the proof.
\end{proof}

With this result at hand, we can prove a quantitative estimate which involves the measure of $\Omega^- \setminus \mathcal Q(\widehat{\Omega})$.

\medskip

We fix a direction $e \in \mathbb{S}^{n-1}$. Without loss of generality, we can assume that $e=e_1$ and that the associated critical hyperplane is $T = \{x_1 = 0\}$, with $\mathcal Q: \mathbb{R}^n \to \mathbb{R}^n, x \mapsto x'$ the reflection with respect to $T$. For the proof of the next lemma we will use the following notation: we set for $t \geq 0$
\begin{equation*}
    \Omega_t \coloneqq \Omega + B_t(0), \quad \widehat{\Omega_t} \coloneqq \Omega_t \cap H^+, \quad 
    \Omega_t^- \coloneqq \Omega_t\cap H^- \quad U_t \coloneqq \mathcal Q(\widehat{\Omega_t}).
\end{equation*}
Note that $G = \Omega_R$.

\begin{lemma}
\label{a2s3lemma1}
Given $P \in \overline{U_R}$ with $B=B_{R/8} (P)$ such that $\mathrm{dist}(B,\partial U_0) \ge R/8$, for $\delta > 0$, we have that
\begin{equation}
\label{a2s3eq39}
|\Omega^- \setminus \mathcal Q(\widehat{\Omega}) \ | \ \leq \Tilde{C} \, (\delta^{-(1+s)} v(P) + \delta),
\end{equation}
where $\Tilde{C} > 0$ is a constant depending only on $n$, $s$, $R$, $\mathfrak{r}^e_\Omega$ and $\mathrm{diam}(\Omega)$.
\end{lemma}

\begin{proof}
We set $K_\delta := (\Omega^- \setminus \mathcal Q(\widehat{\Omega})) \setminus (E_\delta \cup F_\delta)$,
where
$$
E_\delta := \mathcal Q (A_\delta) \cap (\Omega^- \setminus \mathcal Q(\widehat{\Omega})) \quad \text{ with } A_\delta := \{ x \in \Omega^c \ | \ \dist (x, \partial \Omega) < \delta  \},  
$$
$$
F_\delta:= \{ x \in \Omega^- \setminus \mathcal{Q}(\widehat{\Omega}) \, | \, \mathrm{dist}(x, T) < \delta  \}.
$$

Using Lemma \ref{s3lemma2} with $B := B_{R/8} (P)$ and $K := K_\delta$ we obtain
\begin{equation*}
	v \geq \overset{\star}{C} \,  \big[  \mathrm{dist} (K_\delta ,H^+) \, | K_\delta | \, \inf_{K_\delta} v \big] \psi_B \quad \text{in} \ B,
\end{equation*}
where $\overset{\star}{C} > 0$ is an explicit constant depending on $n$, $s$, $R$ and $\diam(\Omega)$. Here we used that, in the present situation, we have $K \subset \Omega$ and that $\dist (B, U_0) \le R$.

Since $K_\delta \subseteq (\Omega^- \setminus \mathcal  Q(\widehat{\Omega})) \setminus F_\delta$, then 
\begin{equation*}
\mathrm{dist} (K_\delta,H^+) \geq \delta .
\end{equation*}

We now point out that in $K_\delta$ we have $v = u - u' = 1 - u'$; we can therefore apply Lemma \ref{lemma_conticino} and get
\begin{equation*}
    v(x) \geq C_{cap} (n,s,\mathfrak{r}^e_{\Omega}) \, \delta^s =  C_{cap} \, \delta^s \quad \text{for every } \ x\in K_\delta.
\end{equation*}

Moreover, we have
\begin{equation}
\label{a2s3eq43}
|K_\delta| = |\Omega^- \setminus U_R| - |E_\delta \cup F_\delta| \ge |\Omega^- \setminus U_R| - (|E_\delta| + | F_\delta |) .
\end{equation}
By definition of $F_\delta$, we have that
\begin{equation}
\label{a2s3eq44}
	|F_\delta| \le  \mathrm{diam}( \Omega )^{n-1} \delta.
\end{equation}
By using Lemma 5.2 in \cite{ciraolo2021symmetry}, since $E_\delta \subseteq A_\delta$, we have
\begin{equation}
\label{a2s3eq45}
    |E_\delta| \leq \left[\frac{2 n |\Omega|}{R} \right]  \delta.
\end{equation}
Putting together \eqref{a2s3eq43}, \eqref{a2s3eq44} and \eqref{a2s3eq45} we get
\begin{equation*}
	|K_\delta| \ge |\Omega^- \setminus U_R| - \tilde{c} \,  \delta,
\end{equation*}
where $\tilde{c}$ is a positive constant depending on $n$, $\mathrm{diam}(\Omega)$ and $\mathfrak{r}^e_{\Omega}$.

\medskip

Hence we have proved that
\begin{equation*}
    v(P) \geq \overset{\star}{C} \, C_{cap} \, (R/8)^{2s} \, \gamma_{n,s} \delta^{1+s} \big( |\Omega^- \setminus U_R| - \tilde{c} \, \delta \big) \,,
\end{equation*}
and, by choosing
\begin{equation*}
    \Tilde{C} := \max \left\lbrace \frac{8^{2s}}{ C_{cap} \, R^s \, \overset{\star}{C}} \, , \, \, \tilde{c} \right\rbrace ,
\end{equation*}
we get the desired inequality \eqref{a2s3eq39}.
\end{proof}

Once Lemma \ref{a2s3lemma1} is proved we can follow \cite{ciraolo2021symmetry} to get the almost symmetry in one direction and, with the same reasoning as in \cite[Section 6]{ciraolo2021symmetry}, obtain the same quantitative stability estimate required for the proof of Theorem \ref{a2theorem2}. For this reason we only sketch the main ideas in the following proof.

\begin{proof}[Proof of Theorem \ref{a2theorem2}]
Once we have inequality \eqref{a2s3eq39}, we can argue as in Lemma 5.6 of \cite{ciraolo2021symmetry} to obtain the estimate for the almost symmetry in one direction, namely
\begin{equation}
\label{a2s3eq49}
    |\Omega^- \setminus \mathcal Q(\widehat{\Omega}) | \leq \overline{C} [u]_{\partial G}^{\frac{1}{s+2}},
\end{equation}
where $\overline{C}:= \max\{ 1, \diam(\Omega), K \, (R/2) \} \, \tilde{C}$, with $\tilde{C}$ as in \eqref{a2s3eq39} and $K = K(n,s) \geq 1$ is the constant that appears in the boundary Harnack inequality in Lemma \ref{a2s1lemma5}.

\medskip

Now, up to a translation we can assume that the critical hyperplanes with respect to the $N$ coordinate directions  $T^{e_j}$ coincide with $\{ x_j = 0 \}$ for each $j=1, \dots , N$, that is, they all intersect at the origin. 

\medskip

The idea is then the following: for a given direction $e \in \mathcal{S}^{n-1}$ we slice $\Omega_{\lambda_e}$ in (a finite number of) sections depending on the critical value $\lambda_e$ and consider their measure, namely
\begin{equation*}
    m_k := | \, \{ x \in \Omega \ | \ (2k-1) \, \lambda_e \leq x \cdot e \leq (2k+1) \, \lambda_e \} \, |, \quad \text{for } k\geq 1.
\end{equation*}

Since $\Omega$ is bounded, $m_k > 0$ only up to an index $k_0$ which behaves like the inverse of $\lambda_e$. The key observation is that, by reflecting with respect to the origin and using \eqref{a2s3eq49}, one has
\begin{equation}
\label{a2s3eq50}
    m_1 = | \, \{ x \in \Omega \ | \ -\lambda_e \leq x \cdot e \leq \lambda_e \} \, | \leq (n+3) \, \overline{C} \, [u]_{\partial G}^{\frac{1}{s+2}};
\end{equation}
moreover, by the moving plane procedure, $m_k \leq m_1$ for every $k$ up to $k_0$ and therefore one can then write the expression 
\begin{equation}
\label{unaltraequazione}
    |\Omega_{\lambda_e}| \leq \sum_{k=1}^{k_0} m_k \leq k_0 \, m_1 \leq (n+3) \, \diam(\Omega) \, \overline{C} \, \frac{1}{\lambda_e} \, [u]_{\partial G}^{\frac{1}{s+2}}.
\end{equation}

Inequalities \eqref{unaltraequazione} and some further calculations (see section 4 in \cite{ciraolo2021symmetry}) yield
\begin{equation}
\label{a2s3eq52}
    |\lambda_e| \leq 4 \, (n+3) \, \frac{\diam(\Omega)}{|\Omega|} \, \overline{C} [u]_{\partial G}^{\frac{1}{s+2}}.
\end{equation}

Now it remains to establish a relationship between $|\lambda_e|$ and $\rho(\Omega)$. We set $\rho_{min}:= \min_{z \in \partial \Omega} |z|$, $\rho_{max}:= \max_{z \in \partial \Omega} |z|$ and choose $x, y \in \partial \Omega$ such that $|x| = \rho_{min}$ and $|y| = \rho_{max}$. We then consider the unit vector
\begin{equation*}
    e := \frac{x - y}{|x - y|}
\end{equation*}
and the corresponding critical hyperplane $T^e$. By construction, we know that $\dist (x,T^e) \geq \dist (y,T^e)$ and therefore some simple calculations lead to 
\begin{equation}
\label{a2s3eq53}
    \rho(\Omega) \leq \rho_{max} - \rho_{min} = |y| - |x| \leq 2 \, \dist (0,T^e) = 2 |\lambda_e|.
\end{equation}

Combining \eqref{a2s3eq52} and \eqref{a2s3eq53} leads to \eqref{a2s1eq8}. It is worth pointing out that the new constants appearing in \eqref{a2s3eq50} and onward only depend on the dimension $n$, the diameter $\diam(\Omega)$ and the volume $|\Omega|$.
\end{proof}

\section{Annular sets}
In this section we consider annular domains and prove Theorems \ref{a2theorem3} and \ref{a2theorem4}. The strategy that we use is still via the method of moving planes and it is similar to the previous one; nevertheless, the method has to be carefully adapted to this situation. We recall that we are considering solutions to \eqref{a2s1eq12}, where $A=\Omega \setminus \overline D$, with $\overline D \subset \Omega$ bounded open domains.

\bigskip

Now for a fixed direction $e$ and a parameter $\lambda \in \mathbb{R}$ we let $T_\lambda$, $H_\lambda$, $\mathcal Q_\lambda$ be as in the previous section. We now consider 
$$
\Sigma_\lambda := (\Omega \cap H_\lambda) \setminus \mathcal Q_\lambda(\overline{D})
$$
which is the cap of the annulus $\Omega \setminus D$. Moreover, for a given set $E$, we define 
\begin{align*}
    &d_{E}:= \inf \{ \lambda \in \mathbb{R} \ | \ T_\mu \cap \overline{E} = \emptyset  \}\\
    &\overline{\lambda}_E := \inf \{ \lambda \leq d_E \ | \ \text{for every } \ \mu > \lambda, (\overline{E} \cap H_\mu)^\mu \subset ( E \cap H_\mu^\mu ) \  \text{ and } \ \nu(x) \cdot e > 0 \ \forall x \in T_\mu \cap \partial E \},
\end{align*}
and the critical parameter $\overline \lambda$ is given by
\begin{equation*}
\overline{\lambda} := \max \{ \overline{\lambda}_D, \ \overline{\lambda}_\Omega  \}.
\end{equation*}

We mention that  both the function $u$ and its reflection $u'$ are $s$-harmonic in $\Sigma_\lambda$, as we are going to use in in the proof of Theorem \ref{a2theorem3}. We also notice that, thanks to our choices, $\overline{\lambda}$ is the critical value for $A$ with respect to the direction $e$, and now the critical position can occur in four possible cases (namely, Cases 1 and 2 in Subsection \ref{subsect_MMP} for both $D$ and $\Omega$).

\medskip

In order to avoid further technicalities we ask for the domains $D$ and $\Omega$ to be regular (namely, with boundaries $\partial G$ and $\partial \Omega$ of class $C^2$); the proof works in the same way if we instead assume that $\partial A$ is just of class $C^1$, and $\Gamma_R^G$ and $\Gamma_R^\Omega$ of class $C^2$.

\medskip

\subsection{Symmetry result}

With this setting, we are now ready to give a proof of the symmetry result for annular sets.

\begin{proof}[Proof of Theorem \ref{a2theorem3}]
We fix a direction $e=e_1$ and reach the critical value $\overline{\lambda}$. Without loss of generality, we assume that $T = \{ e_1 = 0 \}$ and define the function $w(x) := u(x) - u(x')$ for every $x \in \mathbb{R}^n$. To simplify the notation we set $\mathcal{Q} = \mathcal{Q}_{\overline{\lambda}}$. Our aim is to show that $w$ is actually identically zero in $\mathcal Q (\Sigma_{\overline{\lambda}})$. This implies that both the function $w$ and the set $A$ itself are symmetric with respect to direction $e$; since the direction $e$ can be chosen arbitrarily, the proof is then complete.

Hence we have to show that $w \equiv 0$ in $\mathcal Q(\Sigma_{\overline{\lambda}})$. We notice that the function $w$ is antisymmetric with respect to $e=e_1$ and
\begin{align*}
    (-\Delta)^s w(x) = 0 \qquad \qquad &\text{for} \ x \in \mathcal Q(\Sigma_{\overline{\lambda}}),\\
    w(x) = u(x) - u(x') = 1 - u(x') \geq 0 \qquad \qquad &\text{for} \ x \in \overline{D} \cap \mathcal Q(\widehat{\Omega}),\\
    w(x) = u(x) - u(x') = u(x) \geq 0 \qquad \qquad &\text{for} \ x \in \Omega^- \setminus \mathcal Q(\widehat{\Omega}),\\
    w(x) = 0 \qquad \qquad &\text{for} \ x \in H^- \setminus \Omega^-.
\end{align*}
In particular, the last three inequalities tell us that $w \geq 0$ in $H^- \setminus \mathcal Q(\Sigma_{\overline{\lambda}})$. The weak maximum principle for antisymmetric solutions in Lemma \ref{a2s1lemma1} implies that $w \geq 0$ in $\mathcal Q(\Sigma_{\overline{\lambda}})$; then, from the strong maximum principle in Lemma \ref{lemma_SMP} we get that either $w > 0$ in $\mathcal Q(\Sigma_{\overline{\lambda}})$ or $w \equiv 0$ in $\mathcal Q(\Sigma_{\overline{\lambda}})$.

In order to conclude, we notice that in this case we have four possible critical cases, all of which can be treated as in the proof of Theorem \ref{a2s1theorem1}. The conclusion then follows straightforwardly. 
\end{proof}

\subsection{Almost symmetry in one direction}

In order to prove almost symmetry in one direction for the annular set, we need to make use again of the quantitative Hopf's type lemma (\cite[Lemma 4.1]{ciraolo2021symmetry}) and adapt it to the current problem. 

We start with a lemma which gives the behaviour of the solution $u$ of \eqref{a2s1eq12} inside the annulus $A$  with respect to the distance from the boundary. We start with a simple remark.

\begin{remark}
If $u \in C^s (\mathbb{R}^n)$ solves \eqref{a2s1eq12} with $\partial A \in C^1$, then we have
\begin{equation}
\label{a2s4eq3}
    0 < u < 1 \qquad \text{in} \ A.
\end{equation}
\end{remark}
Indeed, applying the maximum principles for an $s$-harmonic function in $A$ we get that $u$ has to be strictly positive in $A$. By using the same argument for $\Tilde{u} := 1 - u$ we get the latter part of \eqref{a2s4eq3}.

\bigskip

We have the following lemma.

\begin{lemma}
\label{a2s4lemma4}
Under the assumptions of Theorem \eqref{a2theorem4}, let $u$ be a solution of \eqref{a2s1eq12}; then
\begin{equation}
\label{a2s4eq36}
    \min \{u, 1 - u \} (x) \geq C^* \, (\dist (x,\partial A))^s \quad \text{in} \ A,
\end{equation}
where
\begin{equation*}
    C^* :=  \frac{c_{n,s} \, \gamma_{n,s}}{4^{n+2s+1}} \, \frac{|D| \, \min \{ \mathfrak{r}^i_\Omega, \mathfrak{r}^e_D, \overline{d}/2\}^s}{\diam (\Omega)^{n+2s}}.
\end{equation*} 
\end{lemma}

\begin{proof}
We first prove \eqref{a2s4eq36} for the function $u$. This inequality follows by an application of Lemma \ref{a2s3theorem1} and therefore we fix a compact set $K_1 \subset D$ such that $|K_1| = |D|/4$. Since $K_1$ is compact and $D$ is open, then $\dist(K_1, \partial D) > 0$.

Let $x \in A$. Assume $\dist(x,\partial A) = \dist (x, \partial \Omega)$ and let $\overline{x} \in \partial \Omega$ be such that $\dist(x,\partial \Omega) = \dist(x,\overline{x}) =: r$. We apply Lemma \ref{a2s3theorem1} with $K=K_1$ and $B=B_r(x)$ and get
\begin{equation}
\label{a2s4eq71}
    u \ge c_{n,s} \, \frac{|K_1| \, \inf_{K_1} u}{(2r + \dist(K_1,B_r(x)) + \diam(K_1))^{n+2s}} \, \psi_{B_r(x)} \qquad \text{in} \ B_r (x).
\end{equation}

Since $u=1$ in $K_1$ and 
$$
2r + \dist(K_1,B_r(x)) + \diam(K_1) \leq 4 \diam(\Omega) \,,
$$ 
by evaluating \eqref{a2s4eq71} at $x$ we have
\begin{equation}
\label{a2s4eq76}
    u(x) \ge \frac{c_{n,s} \, \gamma_{n,s}}{4^{n+2s+1}} \, \frac{|D|}{\diam (\Omega)^{n+2s}} \, r^{2s} =  \frac{c_{n,s} \, \gamma_{n,s}}{4^{n+2s+1}} \, \frac{|D|}{\diam (\Omega)^{n+2s}} \, \dist(x,\partial \Omega)^{2s}.
\end{equation}

Up until now we didn't make use of the interior radius of the touching ball condition $\mathfrak{r}^i_\Omega > 0$. If $\dist (x,\partial \Omega) \ge \mathfrak{r}^i_\Omega$ the equation \eqref{a2s4eq76} immediately gives \eqref{a2s4eq36}. If instead $\dist (x,\partial \Omega) < \mathfrak{r}^i_\Omega$ we set $r_1:= \min \{ \mathfrak{r}^i_\Omega, \, \overline{d}/2 \}$, $\tilde{x} \in A$ such that $\overline{x} \in \partial \Omega \cap \partial B_{r_1} (\tilde{x})$ and apply Lemma \ref{a2s3theorem1} for $K = K_1$ and $B = B_{r_1}(\tilde{x})$ to get
\begin{equation}
\label{a2s4eq79}
    u(x) \ge \frac{c_{n,s} \, \gamma_{n,s}}{4^{n+2s+1}} \, \frac{|D|}{\diam (\Omega)^{n+2s}} \, (r_1 + |x - \tilde{x}|)^s \, (r_1 - |x - \tilde{x}|)^s,
\end{equation}
which together with the fact that $r_1 - |x - \tilde{x}| = \dist(x, \partial \Omega)$ gives \eqref{a2s4eq36}.

\medskip

Assuming now $\dist(x,\partial A) = \dist (x, \partial D)$ we can now repeat the same arguments with $r_2:= \min \{ \mathfrak{r}^e_D, \, \overline{d}/2 \}$ in place if $r_1$ and we obtain \eqref{a2s4eq79} with $r_1$ replaced by $r_2$.

\bigskip

The proof of \eqref{a2s4eq36} for $v:= 1 - u$ can be carried out in the same way; we only have to fix a compact set $K_2 \subset \mathbb{R}^n \setminus \overline \Omega$ such that 
$$
\dist(K_1,B_r(x)) + \diam(K_1) \leq 2 \diam(\Omega) \,,
$$ 
and 
$|K_2| = |D|/4$, and then apply Lemma \ref{a2s3theorem1} with $K = K_2$ as done before. 
\end{proof}

We are now ready to state a version of Lemma \ref{a2s3lemma1} for the annulus, under the assumptions of Theorem \ref{a2theorem4}. The ball $B$ will be chosen inside of a set where the antisymmetic function $w:= u - u'$ is $s$-harmonic (this time, it will be $Q(\Sigma_{\overline{\lambda}})$); in this case, the compact set $K$ consists of two components, in such a way that we can take into account the symmetric differences between the sets $\Omega$ and $G$ and their respective reflections via the moving plane method. We define
\begin{align*}
    &L := \Omega^- \setminus \mathcal Q(\widehat{\Omega}) \\
    &M := D^- \setminus \mathcal Q(\widehat{D})\\
    &\Tilde{K} := L \cup (M \cap \mathcal Q(\widehat{\Omega})).
\end{align*}

We have the following lemma.

\begin{lemma}
Given $P \in \mathcal Q(\Sigma_{\overline{\lambda}})$ with $B=B_{R/8} (P)$ such that $\mathrm{dist}(B, \mathcal Q(\Sigma_{\overline{\lambda}})) \ge R/8$ and given $\delta > 0$, we have that
\begin{equation}
\label{a2s4eq52}
|\Tilde{K}| \leq \Tilde{C} \, (\delta^{-(1+s)} v(P) + \delta),
\end{equation}
where $\Tilde{C} > 0$ is a constant depending only on $n$, $s$, $R$, $\mathfrak{r}^e_{D}$, $\mathfrak{r}^i_{\Omega}$, $\diam(\Omega)$, $|D|$ and $|\Omega|$.
\end{lemma}

\begin{proof}
We set $K_\delta := \Tilde{K} \setminus (E_\delta \cup F_\delta)$ where 
\begin{align*}
    &E_\delta := [ \, \{ x \in \Omega \ | \ \dist(x,\partial \Omega) < \delta \} \cap (\Omega^- \cup \mathcal Q(\widehat{\Omega}) ) \, ] \, \cup \, [ G^- \setminus \mathcal Q(\widehat{G} + B_\delta) ) \, ],\\
    &F_\delta := \{ x \in \Omega \ | \ \dist (x,H^+) < \delta \}.
\end{align*}
We apply Lemma \ref{s3lemma2} with $B := B_{R/8} (P)$ and $K := K_\delta$ to get \eqref{s3eq19}, i.e.

\begin{equation}
\label{a2s4eq55}
	v \geq \overset{\star}{C} \,  \big[  \mathrm{dist} (K_\delta ,H^+) \, | K_\delta | \, \inf_{K_\delta} v \big] \psi_B \quad \text{in} \ B.
\end{equation}

By arguing as done for \eqref{a2s3eq43}, \eqref{a2s3eq44} and \eqref{a2s3eq45} we get that
\begin{equation}
\label{a2s4eq56}
    |\tilde{K}_\delta| \geq |\Omega^- \setminus \mathcal Q(\widehat{\Omega})| + |G^- \setminus \mathcal Q(\widehat{G})| - \tilde{c}\, \delta \,,
\end{equation}
and plugging \eqref{a2s4eq56} into \eqref{a2s4eq55}, together with Lemma \ref{a2s4lemma4} and the fact that $\dist(\tilde{K}_\delta, H^+) \geq \delta$ we get \eqref{a2s4eq52}.
\end{proof}

The way to use Lemma \ref{a2s4lemma4} to establish almost symmetry in one direction and then prove Theorem \ref{a2theorem4} is again the one sketched in the proof of Theorem \ref{a2theorem2}. We just need to highlight some minor differences with the annular case, that we report below.

\begin{proof}
The first goal is to obtain the almost symmetry in one direction from \eqref{a2s4eq52}. While in the proof of Theorem \ref{a2theorem2} we need to take into account the two possible critical Cases 1 and 2 for the moving plane method, with the first one being further divided into cases 1a and 1b, now the critical position can be reached for both the set $D$ and the set $\Omega$, resulting in a total of six possible critical cases. Nonetheless, they are tackled in the same exact way; the only thing that we need to point out is that in each of the critical cases we can write
\begin{equation}
\label{a2s4eq58}
    v(P) \leq  c_\star  \max \{ [u]_{\Gamma^R_{\Omega}}, [u]_{\Gamma^R_{G}} \} = c_\star \, \mathrm{def}_A(u) \,,
\end{equation}
where $c_\star := \max\{1, \diam(\Omega), K \, R/2 \}$.
From \eqref{a2s4eq58} we can then recover the inequality
\begin{equation}
    |\Omega^- \setminus \mathcal Q(\widehat{\Omega})| + |G^- \setminus \mathcal Q(\widehat{G})| \leq \overline{C} \, \mathrm{def}_A(u)^{\frac{1}{s+2}},
\end{equation}
where $\overline{C} = c_\star \tilde{C}$. The slicing of the two sets can then be performed in the same way, which leads to an estimate of type
\begin{equation}
\label{a2s4eq60}
    |\lambda_e| \leq 4 \, (n+3) \frac{\diam(\Omega)}{|\Omega|} \, \overline{C} \, \mathrm{def}_A(u)^{\frac{1}{s+2}},
\end{equation}
where now again the bound depends on the seminorms on both of the parallel surfaces. We now only need to make sure that formula \eqref{a2s3eq53} still applies. Again, for the set $\Omega$ we define $\rho_{min}:= \min_{z \in \partial \Omega} |z|$, $\rho_{max}:= \max_{z \in \partial \Omega} |z|$, choose $x, y \in \partial \Omega$ such that $|x| = \rho_{min}$ and $|y| = \rho_{max}$ and consider the direction $e= y - x$ up to normalization with its critical hypeplane $T^e$. Since we are now in the annular case $\overline{\lambda}^e = \max \{ \overline{\lambda}^e_D, \overline{\lambda}^e_\Omega \}$ and therefore the moving plane might stop before reaching the cricial position for the set $\Omega$ itself. However we can still write
\begin{equation}
\label{a2s4eq62}
    \rho(\Omega) \leq \rho_{max} - \rho_{min} = |y| - |x| \leq 2 \, \dist (0,T^e) = 2 |\overline{\lambda}^e| \leq 2 \, |\overline{\lambda}^e_\Omega|.
\end{equation}
Combining \eqref{a2s4eq60} and \eqref{a2s4eq62} and repeating the same argument for $G$ lead us to \eqref{a2s1eq18}.
\end{proof}

\bibliographystyle{alpha}
\bibliography{biblio}

\newcommand{\etalchar}[1]{$^{#1}$}
\begin{thebibliography}{CFMN18}

\bibitem[ABR99]{aftalion1999approximate}
Amandine Aftalion, J\'{e}r\^{o}me Busca, and Wolfgang Reichel.
\newblock Approximate radial symmetry for overdetermined boundary value
  problems.
\newblock {\em Adv. Differential Equations}, 4(6):907--932, 1999.

\bibitem[CDP{\etalchar{+}}23]{ciraolo2021symmetry}
Giulio Ciraolo, Serena Dipierro, Giorgio Poggesi, Luigi Pollastro, and Enrico
  Valdinoci.
\newblock Symmetry and quantitative stability for the parallel surface
  fractional torsion problem.
\newblock {\em To appear in Trans. Amer. Math. Soc.}, 2023.

\bibitem[CFMN18]{ciraolo2018rigidity}
Giulio Ciraolo, Alessio Figalli, Francesco Maggi, and Matteo Novaga.
\newblock Rigidity and sharp stability estimates for hypersurfaces with
  constant and almost-constant nonlocal mean curvature.
\newblock {\em J. Reine Angew. Math.}, 741:275--294, 2018.

\bibitem[CMS15]{CMS}
Giulio Ciraolo, Rolando Magnanini, and Shigeru Sakaguchi.
\newblock Symmetry of minimizers with a level surface parallel to the boundary.
\newblock {\em J. Eur. Math. Soc. (JEMS)}, 17(11):2789--2804, 2015.

\bibitem[CMS16]{ciraolo2016solutions}
Giulio Ciraolo, Rolando Magnanini, and Shigeru Sakaguchi.
\newblock Solutions of elliptic equations with a level surface parallel to the
  boundary: stability of the radial configuration.
\newblock {\em J. Anal. Math.}, 128:337--353, 2016.

\bibitem[CV18]{CiraoloVezzoni}
Giulio Ciraolo and Luigi Vezzoni.
\newblock A sharp quantitative version of {A}lexandrov's theorem via the method
  of moving planes.
\newblock {\em J. Eur. Math. Soc. (JEMS)}, 20(2):261--299, 2018.

\bibitem[DNPV12]{di2012hitchhiker}
Eleonora Di~Nezza, Giampiero Palatucci, and Enrico Valdinoci.
\newblock Hitchhiker's guide to the fractional {S}obolev spaces.
\newblock {\em Bull. Sci. Math.}, 136(5):521--573, 2012.

\bibitem[FJ15]{fall2015overdetermined}
Mouhamed~Moustapha Fall and Sven Jarohs.
\newblock Overdetermined problems with fractional {L}aplacian.
\newblock {\em ESAIM Control Optim. Calc. Var.}, 21(4):924--938, 2015.

\bibitem[GS16]{greco2016hopf}
Antonio Greco and Raffaella Servadei.
\newblock Hopf’s lemma and constrained radial symmetry for the fractional
  laplacian.
\newblock {\em Mathematical Research Letters}, 23(3):863--885, 2016.

\bibitem[ROS14]{ros2014dirichlet}
Xavier Ros-Oton and Joaquim Serra.
\newblock The {D}irichlet problem for the fractional {L}aplacian: regularity up
  to the boundary.
\newblock {\em J. Math. Pures Appl. (9)}, 101(3):275--302, 2014.

\bibitem[Ser71]{serrin1971symmetry}
James Serrin.
\newblock A symmetry problem in potential theory.
\newblock {\em Arch. Rational Mech. Anal.}, 43:304--318, 1971.

\bibitem[SV19]{soave2019overdetermined}
Nicola Soave and Enrico Valdinoci.
\newblock Overdetermined problems for the fractional laplacian in exterior and
  annular sets.
\newblock {\em Journal d'Analyse Math{\'e}matique}, 137(1):101--134, 2019.

\bibitem[War15]{warma2015fractional}
Mahamadi Warma.
\newblock The fractional relative capacity and the fractional laplacian with
  neumann and robin boundary conditions on open sets.
\newblock {\em Potential Analysis}, 42(2):499--547, 2015.

\end{thebibliography}

\end{document}